\documentclass[12pt]{amsart}
\usepackage{amsmath,amsthm,amsfonts,latexsym,amssymb,amscd,color}
\usepackage{chemarr}
\newtheorem{thm}{Theorem}[section] 
\newtheorem{cor}[thm]{Corollary}
\newtheorem{lm}[thm]{Lemma}

\newtheorem{clm}[thm]{Claim}
\newtheorem*{clm*}{Claim}
\theoremstyle{definition}

\newtheorem{remark}[thm]{Remark}

\numberwithin{equation}{section}
\newcommand{\cproof}{\noindent{\it Proof of Claim.}\ } 
\newcommand{\cqed}{\hfill\rule{1.3mm}{3mm}}

\DeclareMathOperator{\Hom}{\mathsf{H}}

\DeclareMathOperator{\Sub}{\mathsf{S}}

\DeclareMathOperator{\Prod}{\mathsf{P}}

\newcommand{\wec}[1]{{\mathbf{#1}}}  

\newcommand{\m}[1]{{\mathbf{\uppercase{#1}}}}

\DeclareMathOperator{\Pol}{Pol}
\DeclareMathOperator{\Con}{Con}

\begin{document}

\title[Minimal abelian varieties of algebras, I]{Minimal abelian varieties of algebras, I}

\author{Keith A. Kearnes}
\address[Keith A. Kearnes]{Department of Mathematics\\
University of Colorado\\
Boulder, CO 80309-0395\\
USA}
\email{kearnes@colorado.edu}

\author{Emil W. Kiss}
\address[Emil W. Kiss]{
Lor\'{a}nd E{\"o}tv{\"o}s University\\
Department of Algebra and Number Theory\\
H--1117 Budapest, P\'{a}zm\'{a}ny P\'{e}ter s\'{e}t\'{a}ny 1/c.\\
Hungary}
\email{ewkiss@cs.elte.hu}

\author{\'Agnes Szendrei}
\address[\'Agnes Szendrei]{Department of Mathematics\\
University of Colorado\\
Boulder, CO 80309-0395\\
USA}
\email{szendrei@colorado.edu}

\thanks{This material is based upon work supported by
  the National Science Foundation grant no.\ DMS 1500254,
  the Hungarian National Foundation for Scientific Research (OTKA)
  grant no.\ K115518, and
  the National Research, Development and Innovation Fund of Hungary (NKFI)
  grant no.\ K128042.}

\subjclass[2010]{Primary: 03C05; Secondary: 08A05, 08B15}
\keywords{Abelian algebra, minimal variety,
  rectangular algebra, strongly abelian algebra, term condition}

\begin{abstract}
  We show that any abelian variety
  that is not affine has a nontrivial strongly abelian subvariety.
  In later papers in this sequence we apply this result
  to the study of minimal abelian varieties.
\end{abstract}

\maketitle

\section{Background}\label{back}
This paper concerns the classification of minimal
varieties of algebras.
Early work on this topic focused on determining
the minimal subvarieties of known varieties
(e.g., groups, rings, modules, lattices, etc). For a survey
of these results, please consult \cite{szendrei-survey}.
This 1992 survey also contains the state of knowledge
(at the time) of the problem of classifying all minimal
locally finite varieties.

But shortly after the publication of \cite{szendrei-survey},
two groups of researchers (Szendrei on the one hand
and Kearnes-Kiss-Valeriote on the other)
independently classified the 
minimal, locally finite, abelian varieties of algebras.
Multiple proofs of the classification theorem were discovered
and presented in the papers
\cite{kearnes-kiss-valeriote,kearnes-szendrei1,szendrei1,szendrei2}.
Those proofs start with the observation that each
minimal locally finite variety contains
a smallest nontrivial member, which must be
finite, simple, and have no proper nontrivial subalgebras.
Tame congruence theory assigns a number (a {\it type})
to any finite simple
algebra: if abelian, the type must be {\bf 1} (the $G$-set type),
or type {\bf 2} (the vector space type). The type {\bf 1}/type {\bf 2}
case division is the main case division in the classification
of minimal, locally finite, abelian varieties.
Within each of these two cases there are subcases
related to dimension and to the field associated to the vector
space in type {\bf 2}. The full classification is accomplished
by examining type~{\bf 1} and type~{\bf 2}
simple algebras until one can isolate out and fully describe
those which generate minimal varieties.

The problem of extending these results to 
varieties that are not locally finite
has been considered, but only under additional
very strong hypotheses.
For example, in \cite[Theorem~5.12]{kearnes-id-simple}
the classification of minimal abelian varieties
is obtained under the assumption that the variety is idempotent
and contains a nontrivial quasiaffine algebra.
In \cite[Corollary~2.10]{kearnes-min-id}
this is generalized to eliminate the hypothesis
that the variety
contains a nontrivial quasiaffine algebra.
The result here is a classification
of arbitrary minimal, abelian, idempotent varieties

\bigskip

This is the first in a sequence of three papers in which we
attempt to classify the minimal abelian varieties
without any additional assumptions at all.
We predict that a complete
classification proof might evolve along these lines:

\bigskip

\noindent
{\bf Goal 1.} Show that minimal abelian varieties exist in two
unrelated types, corresponding to the type {\bf 1}/type {\bf 2}
case division observed in the locally finite setting.
We propose that these types should be: strongly abelian
minimal varieties as the extension of the type {\bf 1} case,
and affine varieties
as the extension of the type {\bf 2} case.

\bigskip

\noindent
{\bf Goal 2.}
Classify the minimal affine varieties.

\bigskip

\noindent
{\bf Goal 3.}
Classify the minimal strongly abelian varieties.

\bigskip

In this paper we accomplish Goal 1.
Actually, we prove a stronger statement that does not
involve minimality, namely we prove that any abelian variety
that is not affine has a nontrivial strongly abelian subvariety.

In the second paper
in the sequence, \cite{kksz2}, we accomplish Goal 2 
in the following sense: we reduce the classification
of minimal affine varieties to the classification of
simple rings. Each minimal affine variety has an
associated simple ring, and each simple ring is
associated to some minimal affine varieties.
This is a many-to-one correspondence between
minimal affine varieties and simple rings. We completely
explain the relationship between the varieties and the rings.

In the third paper in the sequence, \cite{kksz3}, we make partial progress
on Goal 3. Namely, we classify those minimal strongly abelian
varieties that have a finite bound on the essential arities
of their terms. Here, when we say `we classify',
we mean that we reduce the classification of
these varieties
to the classification
of simple monoids with zero.
Also in the third paper of the sequence we show that
there are minimal strongly abelian varieties
that do not have a finite bound on the
essential arities
of their terms, thereby showing that there is more
work to do to complete the classification of minimal
strongly abelian varieties.
    
\bigskip

\section{Terminology and notation} \label{prelim}

An algebraic language $\mathcal L$ is determined by a function
$\alpha: F\to \omega$ where $F$ is a set of operation symbols
and $\alpha$ assigns arity. An algebra for this language
is a pair $\langle A; F\rangle$ where $A$ is a nonempty set
and for each symbol $f\in F$ there is a fixed interpretation
$f^{\m a}: A^{\alpha(f)}\to A$ of that symbol as an $\alpha(f)$-ary operation
on $A$.

Let $X = \{x_1, \ldots \}$ be a set of variables.
The set $\mathcal T$
of all terms in $X$ in a language $\mathcal L$ is defined
recursively by stipulating that (i) $X\subseteq {\mathcal T}$,
and (ii) if $f\in F$, $\alpha(f)=k$,
and $t_1,\ldots, t_k\in {\mathcal T}$, then
$f(t_1,\ldots,t_k)\in {\mathcal T}$.
The assignment $f\mapsto f^{\m a}$, which assigns
an operation table to a symbol, can be extended to terms
$t\mapsto t^{\m a}$. We call the interpretation $t^{\m a}$
of $t$ the term operation of $\m a$ associated to the term $t$.

An identity in $\mathcal L$ is a pair of terms,
written $s\approx t$. The identity $s\approx t$ is satisfied by $\m a$,
written $\m a\models s\approx t$, if
$s^{\m a}=t^{\m a}$. Given a set $\Sigma$
of identities, the class $\mathcal V$
of all $\mathcal L$-algebras
satisfying $\Sigma$ is called the variety axiomatized
by $\Sigma$.

We shall use Birkhoff's Theorem, which asserts
that the smallest variety containing a class
$\mathcal K$ of $\mathcal L$-algebras is the class
$\Hom\Sub\Prod(\mathcal K)$ of homomorphic
images of subalgebras of products of algebras in $\mathcal K$.

A subvariety of a variety
$\mathcal V$ is a subclass of $\mathcal V$
that is a variety.
A variety is trivial if it
consists of $1$-element algebras only.
A variety is minimal if it is not trivial, but
any proper subvariety is trivial.

The full constant expansion of $\m a$
is the algebra $\m a_A=\langle A; F\cup \{c_a\;|\;a\in A\}\rangle$
obtained from $\m a$
by adding a new $0$-ary (constant) symbol $c_a$
for each element $a\in A$.
By a polynomial operation
of $\m a$ we mean
a term operation of $\m a_A$. 

A $1,1$-matrix of $\m a$ is a $2\times 2$ matrix of elements
of $A$ of the form
\begin{equation} \label{matrixDEFN}
  \left[\begin{array}{cc}
t(\wec{a},\wec{u}) & t(\wec{a},\wec{v}) \\
t(\wec{b},\wec{u}) & t(\wec{b},\wec{v}) 
      \end{array}\right]=
  \left[\begin{array}{cc}
      p&q\\
      r&s
      \end{array}\right]\in A^{2\times 2}
  \end{equation}
  where $t(\wec{x},\wec{y})$ is a polynomial
  of $\m a$ and $\wec{a}, \wec{b}, \wec{u}, \wec{v}$
  are tuples of elements of $A$.
  The set of $1, 1$-matrices is invariant under
  the operations of
  swapping rows, swapping columns, and matrix transpose.

  $\m a$ has property ``$X$'', or ``is $X$'', if 
  the corresponding implications hold for all $1,1$-matrices
  in (\ref{matrixDEFN}):
\begin{itemize}  
\item ($X$ = abelian) provided
  $p=q$ implies $r=s$. (Equivalently, if
  $p=r$ implies $q=s$.)
\item ($X$ = rectangular) provided that,
  for some compatible partial order on $\m a$, $\geq$,
  it is the case that
  $u\geq q$ and $u\geq r$ together imply $u\geq s$. 
\item ($X$ = strongly rectangular) provided $q=r$ implies $r=s$. 
\item ($X$ = strongly abelian) (same as abelian + strongly rectangular).
\item ($X$ = affine) (same as abelian + has a Maltsev operation).
\end{itemize}
A variety ``is $X$'' if all of its algebras are.

All of these concepts will be used in this paper.
What we have just written is not enough to understand
what follows, so please see Chapters 2 and 5
of \cite{kearnes-kiss} for more detail
about these concepts when necessary.

\section{Abelian and affine algebras} \label{abaff}

Our main goal in this section is to prove
that if $\mathcal V$ is an abelian variety that
is not an affine variety, then $\mathcal V$ contains a 
nontrivial strongly abelian subvariety.
Applying this to the situation where $\mathcal V$
is a minimal variety, we obtain that
any minimal abelian variety is affine or strongly abelian.

The path we follow in this section is to prove
the following sequentially stronger
Facts about an arbitrary abelian variety $\mathcal V$ that is not affine:
\begin{enumerate}
\item[(I)] $\mathcal V$ contains an algebra with a nontrivial
  strongly abelian congruence.
\item[(II)] $\mathcal V$ contains a nontrivial
  strongly abelian algebra.
\item[(III)] $\mathcal V$ contains a nontrivial
  strongly abelian subvariety.
\end{enumerate}

The following theorem establishes Fact (I).

\begin{thm} \label{I}
  Let $\mathcal V$ be an abelian variety. If $\mathcal V$ is not affine,
then there is an algebra $\m a\in\mathcal V$ that has a nontrivial
  strongly abelian congruence.
\end{thm}

\begin{proof}
  We prove the contrapositive of the second
  sentence in the theorem statement. Namely, under
  the hypothesis that $\mathcal V$
  is abelian, we show that if $\mathcal V$ contains
  no algebra $\m a$ with a nontrivial
  strongly abelian congruence, then $\mathcal V$ is affine.

If $\mathcal V$ has no algebra $\m a$ with a nontrivial
strongly abelian congruence, then Theorem~3.13 of \cite{kearnes-kiss}
proves that $\mathcal V$ satisfies a nontrivial idempotent
Maltsev condition.
By Theorem~3.21 of
\cite{kearnes-kiss}, any variety that satisfies
a nontrivial idempotent
Maltsev condition
has a \emph{join term}, which is a term
whose associated term operation
acts as a semilattice join operation
on the blocks of any rectangular tolerance relation
of any algebra in the variety.
But no subset of more than one element in an abelian algebra
can be closed under a semilattice term operation, because this would
realize a nontrivial semilattice as a subalgebra
of a reduct of an abelian algebra. Subalgebras of reducts
of abelian algebras are abelian, and no semilattice
of more than one element is abelian. This shows that
blocks of rectangular tolerances in $\mathcal V$ are singleton sets,
which is another way of saying that $\mathcal V$ contains
no algebra with a nontrivial rectangular tolerance.

By Theorem~5.25 of \cite{kearnes-kiss}, the fact
that $\mathcal V$ omits nontrivial rectangular
tolerances is equivalent to the fact that
$\mathcal V$ satisfies
an idempotent Maltsev condition that fails
in the variety of semilattices.
Finally, Theorem~4.10 of \cite{kearnes-szendrei2},
proves that if $\mathcal V$ is
any variety satisfying an idempotent Maltsev
condition which fails in the variety of semilattices,
then abelian algebras in $\mathcal V$ are affine.
Altogether, this shows that if $\mathcal V$
is abelian and no algebra in $\mathcal V$ has a nontrivial
  strongly abelian congruence, then $\mathcal V$ is affine.
  \end{proof}

This concludes the proof of Fact (I).
Our next goal is to prove Fact (II): if
$\mathcal V$ is abelian but not affine,
then $\mathcal V$ contains a nontrivial
strongly abelian algebra.

The following notation will be needed for
Lemma~\ref{nonaffine}, which is a result proved in 
\cite{kearnes-kiss-szendrei} (Lemma~2.1 of that paper).
Assume that $\m a$ is abelian
and $\theta\in\Con(\m a)$ is strongly abelian.
Let $\m a(\theta)$ be the subalgebra of $\m a\times \m a$
supported by the graph of $\theta$.
Let $\Delta$ be the congruence on $\m a(\theta)$ 
generated by $D\times D$ where $D = \{(a,a)\;|\;a\in A\}$
is the diagonal. 
$D$ is a $\Delta$-class, because $\m a$ is abelian.
Let $\m s = \m s_{\m a,\theta} := \m a(\theta)/\Delta$.
Let $0 = D/\Delta\in S$.

\begin{lm}\label{nonaffine}
  Let $\mathcal V$ be an abelian variety, and suppose that
  $\theta$ is a nontrivial strongly abelian congruence on some
  $\m a\in\mathcal V$. Let $\m s = \m s_{\m a,\theta}$
  and let $0 = D/\Delta\in S$. The following are true:
  \begin{enumerate}
    \item $\m s$ has more than one element.
  \item $\{0\}$ is a 1-element subuniverse of $\m s$.
  \item $\m s$ has ``Property P'': for every $n$-ary polynomial
    $p(\wec{x})$ of $\m s$ and every tuple $\wec{s}\in S^n$
    \[
p(\wec{s})=0\quad\textrm{implies}\quad p(\wec{0})=0,
\]

  \bigskip

  \noindent
where $\wec{0} = (0,0,\ldots,0)$.
  \item Whenever $t(x_1,\ldots,x_n)$
    is a $\mathcal V$-term and
    \[
    \mathcal V\models t(\wec{x})\approx 
    t(\wec{y})
    \]

  \bigskip

  \noindent
    where $\wec{x}$ and $\wec{y}$ are tuples of not necessarily
    distinct variables which differ in the $i$th position,
    then the term operation $t^{\m s}(x_1,\ldots,x_n)$ is
    independent of its $i$th variable.
  \item $\m s$ has a congruence $\sigma$ such that the
    algebra $\m s/\sigma$ satisfies (1)--(4) of this lemma,
    and $\m s/\sigma$ also
    has a compatible partial order $\leq$ such that
$0\leq s$ for every $s\in (S/\sigma)$.
    \end{enumerate} \qed
  \end{lm}

This lemma puts us in position to establish Fact (II):

\begin{thm} \label{II}
  If $\mathcal V$ is an abelian variety that is not affine, then
  $\mathcal V$ contains a nontrivial strongly abelian algebra.
  \end{thm}

\begin{proof}
By Theorem~\ref{I}, the assumption that $\mathcal V$
is abelian but nonaffine guarantees that there is some
$\m a\in{\mathcal V}$ that has some
nontrivial strongly abelian congruence.
By Lemma~\ref{nonaffine}, these data can be used
to construct a nontrivial algebra
$\m t:=\m s/\sigma\in{\mathcal V}$
that has a compatible partial order $\leq$
and a singleton subuniverse $\{0\}$ such that $0$ is the least element
under the partial order. We proceed from this point.

Observe that if $a, b\in T$ satisfy $a\geq b$,
then from $a\geq b\geq 0$ we derive that
$f(a)\geq f(b)\geq f(0)\;(\geq 0)$
for any unary polynomial $f\in\Pol_1(\m t)$.
In particular, if
$f(a) = 0$ we must also have $f(b)=0$.
Let us define a coarser quasiorder $\sqsupseteq$
on $\m t$ by this rule: for $a, b\in T$, let
$a\sqsupseteq b$ if 

\begin{equation} \label{implication}
f(a)=0\Rightarrow f(b)=0
\end{equation}
  
\bigskip

\noindent
holds for all $f\in\Pol_1(\m t)$.
The relation $\sqsupseteq$ is reflexive, transitive,
and compatible with unary polynomials,
so it is compatible with all polynomials.
Therefore $\sqsupseteq$ is a compatible
quasiorder on $\m t$. From the first two sentences
of this paragraph we see that
$\sqsupseteq$ extends $\geq$ (i.e., $\sqsupseteq$
is a coarsening of $\geq$).
This is enough to imply that $0$
is a least element with respect to
$\sqsupseteq$.

Let $\theta = \sqsupseteq\cap \sqsupseteq^{\cup}$.
  By considering what happens in (\ref{implication}) when $0\sqsupseteq b$
and $f(x) = x$ one sees that $0\sqsupseteq b$ implies $b=0$.
Hence $0/\theta = \{0\}$, from which it follows that
$\theta$ is a proper congruence of $\m t$.

We let $\m t'=\m t/\theta$,
$0'=0/\theta$, and $\geq' = {\sqsupseteq}/\theta$.
Now $\m t'$ is nontrivial,
has a compatible partial order $\geq'$ with least
element $0'$, and with respect to this partial order we have
$a\geq' b$ if and only if
\[f(a)=0'\Rightarrow f(b)=0'\]
  
  \bigskip

  \noindent
for all unary polynomials $f\in\Pol_1(\m t')$.
  
Our new algebra $\m t'$
is a quotient of the original algebra,
has all properties attributed to $\m t$
in the first paragraph of this proof, but
now we have strengthened the implication
``$a\geq' b$ in $\m t'$ implies $f(a)=0'\Rightarrow f(b)=0'$
for all unary polynomials $f\in\Pol_1(\m t')$''
to a bi-implication.
We now replace $\m t$ with $\m t'$,
drop all primes, and assume that
\begin{equation} \label{new_implication}
\textrm{$a\geq b$ in $\m t$ if and only if $f(a)=0\Rightarrow f(b)=0$.}
\end{equation}
    
It should be pointed out that the
reflexivity, transitivity, and compatibility of $\geq$
implies that $\m t$ satisfies Property P
of Lemma~\ref{nonaffine}. For suppose that $\wec{s}=(s_1,\ldots,s_n)$
and $p(\wec{s})=0$ for some polynomial $p$ of $\m t$.
Since $s_i\geq 0$ for all $i$ we derive from (\ref{new_implication})
that
\[
\begin{array}{rl}
p(\wec{s})= p(s_1,s_2,s_3,\ldots,s_n)=0 & \Rightarrow p(0,s_2,s_3,\ldots,s_n)=0 \\
& \Rightarrow p(0,0,s_3,\ldots,s_n)=0 \\
& \vdots\\
& \Rightarrow p(0,0,0,\ldots,0)=0, \\
\end{array}
\]
which is the assertion of Property P.  

  \begin{clm} \label{rectangulation}
    The total binary relation $1\in\Con(\m t)$ rectangulates itself
    with respect to $\geq$.
  \end{clm}

  \cproof
  Assume that
\begin{equation} \label{rect_matrix}
  \left[\begin{array}{cc}
t(\wec{a},\wec{u}) & t(\wec{a},\wec{v}) \\
t(\wec{b},\wec{u}) & t(\wec{b},\wec{v}) 
      \end{array}\right]=
  \left[\begin{array}{cc}
      p&q\\
      r&s
      \end{array}\right]
\end{equation}

  \bigskip

  \noindent
  is a $1,1$-matrix and that $u\geq q, r$. Our goal is to prove
  that $u\geq s$. If $u\not\geq s$, then
  according to (\ref{new_implication})
  there
  is a unary polynomial $f$ such that $f(u)=0$ and $f(s)\neq 0$.
  Since $u\geq q, r$ and since $0$ is the least element under $\geq$,
  we get that $0=f(u)\geq f(q), f(r) \geq f(0)$,
  so in fact $0=f(u) = f(q) = f(r) = f(0)$.
  Prefixing the polynomial $t$ in the
  left matrix of (\ref{rect_matrix})
  with the polynomial
  $f$, we obtain a $1, 1$-matrix of the form
\begin{equation} \label{11matrix}
  \left[\begin{array}{cc}
ft(\wec{a},\wec{u}) & ft(\wec{a},\wec{v}) \\
ft(\wec{b},\wec{u}) & ft(\wec{b},\wec{v}) 
      \end{array}\right]=
  \left[\begin{array}{cc}
      f(p)&f(q)\\
      f(r)&f(s)
      \end{array}\right]  = 
  \left[\begin{array}{cc}
      p'&0\\
      0&s'
      \end{array}\right]
\end{equation}

  \bigskip

  \noindent
  with $s'\neq 0$. That is,

\begin{equation} \label{crossdiag}
ft(\wec{a},\wec{v}) = 0 =   ft(\wec{b},\wec{u}),
\end{equation}

  \bigskip

  \noindent
  while $ft(\wec{b},\wec{v}) \neq 0$.
  Employing Property P in the forms
  $ft(\underline{\wec{a}},\wec{v}) = 0 \Rightarrow ft(\underline{\wec{0}},\wec{v}) = 0$
  and
  $ft(\underline{\wec{b}},\wec{u}) = 0 \Rightarrow ft(\underline{\wec{0}},\wec{u}) = 0$,
  we get from (\ref{crossdiag}) that

\begin{equation} \label{crossdiag2}  
ft(\underline{\wec{0}},\wec{v}) = 0 =   ft(\underline{\wec{0}},\wec{u}).
\end{equation}

\bigskip

  \noindent
Since $\m t$ is abelian, we derive from (\ref{crossdiag2}) that
\begin{equation} \label{contraequation}
ft(\underline{\wec{b}},\wec{v}) =  ft(\underline{\wec{b}},\wec{u}).
\end{equation}  
The left side of (\ref{contraequation})
equals $s'$ while the right side equals $0$,
yielding $s'=0$, contrary to the 
fact stated
after the line containing (\ref{11matrix}). The claim is proved.
  \cqed

  \bigskip

    \begin{clm}
    The total binary relation $1\in\Con(\m t)$
    strongly rectangulates itself.
  \end{clm}

    \cproof
Assume that
\begin{equation} \label{st_rect_matrix}
  \left[\begin{array}{cc}
t(\wec{a},\wec{u}) & t(\wec{a},\wec{v}) \\
t(\wec{b},\wec{u}) & t(\wec{b},\wec{v}) 
      \end{array}\right]=
  \left[\begin{array}{cc}
      p&q\\
      r&s
      \end{array}\right]
\end{equation}

\bigskip

\noindent  
is a $1,1$-matrix and that $q = r$. Our goal is to prove
that $r = s$. Note that by taking $u = q = r$ we get
from rectangulation with respect to $\geq$
(Claim~\ref{rectangulation}) that, since
$u\geq q, r$, we must have $u\geq s$, so in particular
we have $r\geq s$. If we do not have $r=s$ as desired,
then we must have $s\not\geq r$. In this case there
is a unary polynomial $f$ such that $f(s) = 0$
and $f(r)\neq 0$.
By prefixing the polynomial $t$
in the left matrix in (\ref{st_rect_matrix})
by $f$, we obtain a matrix of the form
\[
\left[\begin{array}{cc}
ft(\wec{a},\wec{u}) & ft(\wec{a},\wec{v}) \\
ft(\wec{b},\wec{u}) & ft(\wec{b},\wec{v}) 
      \end{array}\right]=
  \left[\begin{array}{cc}
      f(p)&f(q)\\
      f(q)&f(s)
      \end{array}\right]=
  \left[\begin{array}{cc}
      p'&q'\\
      q'&0
      \end{array}\right]
  \]
\bigskip

\noindent
with $q'=f(q)=f(r)\neq 0$. We also have (by rectangulation with
respect to $\geq$) that
any $u$ that majorizes the cross diagonal in 
\[
\left[\begin{array}{cc}
  p'&q'\\
  q'&0
\end{array}\right],
\]

\bigskip

\noindent
like $u=q'$, also majorizes the main diagonal, i.e. $q'\geq p'$.
Similarly, any $u$ that majorizes the main diagonal, 
like $u=p'$, also majorizes the cross diagonal, i.e. $p'\geq q'$.
In particular, $p'=q'$ and we have 
\[
\left[\begin{array}{cc}
ft(\wec{a},\wec{u}) & ft(\wec{a},\wec{v}) \\
ft(\wec{b},\wec{u}) & ft(\wec{b},\wec{v}) 
\end{array}\right]=
\left[\begin{array}{cc}
q'&q'\\
q'&0
\end{array}\right].
\]

\bigskip

\noindent
This is a failure of the term condition (which defines abelianness),
thereby proving the claim.
\cqed

\bigskip
  
We complete the proof of Theorem~\ref{II} by noting that
an algebra is strongly
abelian if and only if it is abelian and strongly rectangular.
Since we have shown that $\m t$ is strongly rectangular,
and we selected $\m t$ from an abelian variety, we conclude
that $\m t$ is strongly abelian.
\end{proof}

Next on our agenda is to prove Fact (III),
which asserts that, if $\mathcal V$ is abelian but not affine,
then $\mathcal V$ has a nontrivial subvariety that
is strongly abelian. We shall require the following lemma,
which is an extension of 
\cite[Theorem~7.1]{mckenzie-valeriote}.

Recall 
the class operators $\Hom$, $\Sub$, and $\Prod$
we introduced briefly in Section~\ref{prelim}.
Namely, for a class $\mathcal K$ of similar algebras 
$\Hom(\mathcal K)$ denotes the class of algebras that are
homomorphic images of members of $\mathcal K$, 
$\Sub(\mathcal K)$ denotes the class of algebras isomorphic
to subalgebras of members 
of $\mathcal K$,  and
$\Prod(\mathcal K)$ denotes the class of algebras isomorphic
to products of members 
of $\mathcal K$. Each of the classes
$\Hom(\mathcal K)$, $\Sub(\mathcal K)$, and
$\Prod(\mathcal K)$ has been defined so that it is closed
under isomorphism.

\begin{lm} \label{fake-mck-val}
  Assume that every finitely generated
  subalgebra of $\m a$ is strongly solvable.
  If\/ $\Hom\Sub(\m a^2)$ consists of abelian algebras,
  then $\Hom\Sub(\m a)$ consists of strongly
  abelian algebras.
\end{lm}

\begin{proof}
  For the first step of the proof we invoke 
  \cite[Theorem~7.1]{mckenzie-valeriote}, which proves the following:

\begin{clm} \label{fake-mck-val-clm}
Assume that every finitely generated
  subalgebra of $\m a$ is strongly solvable.
  If\/ $\Hom\Sub(\m a^2)$ consists of abelian algebras,
  then $\m a$ is strongly abelian.
\end{clm}

The conclusion that $\m a$ is strongly abelian in
Claim~\ref{fake-mck-val-clm}
implies that every algebra in
$\Sub(\m a)$ is also strongly abelian, since the
strong abelian property is expressible by
universal sentences. However it is not an
immediate consequence of Claim~\ref{fake-mck-val-clm}
that the class
$\Hom\Sub(\m a)$ consists of strongly abelian algebras.
For this we must show that if $\m b\in \Sub(\m a)$
and $\theta\in\Con(\m b)$, then $\m b/\theta$ is also
strongly abelian.

Choose and fix $\m b\in \Sub(\m a)$
and $\theta\in\Con(\m b)$.
Recall (from Section~\ref{prelim})
that an algebra is strongly abelian if and only if
it is both abelian and strongly rectangular.
We are assuming that all algebras in $\Hom\Sub(\m a^2)$
are abelian, and $\m b/\theta$ is in $\Hom\Sub(\m a)$,
which is a subclass of the abelian class $\Hom\Sub(\m a^2)$,
so to prove the lemma it suffices to prove that $\m b/\theta$
is strongly rectangular.

This will be a proof by contradiction.
Our aim will be to obtain a contradiction
from the assumptions that:
$\m a$ is strongly abelian (which we get from Claim~\ref{fake-mck-val-clm}), 
$\Hom\Sub(\m a^2)$ consists of abelian algebras,
$\m b\leq \m a$, $\theta\in\Con(\m b)$,
but $\m b/\theta$ is not strongly rectangular.
Observe that the assumptions that
$\m a$ is strongly abelian and $\m b\leq \m a$
imply that $\m b$ is strongly abelian.
We reiterate the observation of the last paragraph
that the assumption that
$\Hom\Sub(\m a^2)$ consists of abelian algebras
implies that $\m b/\theta$ is abelian.

The assumption that $\m b/\theta$ is not strongly rectangular
means exactly that there is a $1,1$-matrix in $\m b$
of the form
\begin{equation} \label{matrixEQ}
  \left[\begin{array}{cc}
t(\wec{a},\wec{u}) & t(\wec{a},\wec{v}) \\
t(\wec{b},\wec{u}) & t(\wec{b},\wec{v}) 
      \end{array}\right]=
  \left[\begin{array}{cc}
      p&q\\
      r&s
      \end{array}\right]
\end{equation}
with $q\equiv_{\theta} r$ but $r\not\equiv_{\theta} s$.

\begin{clm} \label{theta-related}
  No two elements of the second matrix of (\ref{matrixEQ})
  which come from the same row or column
are $\theta$-related.
\end{clm}

\cproof
We explain why $p$ and $r$, the two elements
of the first column of the second matrix in (\ref{matrixEQ}),
are not $\theta$-related and omit the proofs of the other (similar)
cases.

Assume that $p\equiv_{\theta} r$.
Since $\m b/\theta$ is abelian, 
the ordinary term condition
holds in $\m b/\theta$, i.e., $C(1,1;\theta)$ holds in $\m b$.
From
\[
t(\wec{a},\underline{\wec{u}}) = p\equiv_{\theta} r = t(\wec{b},\underline{\wec{u}}) 
\]
\bigskip

\noindent
we derive 
\[
t(\wec{a},\underline{\wec{v}}) = q\equiv_{\theta} s= t(\wec{b},\underline{\wec{v}}) 
\]
\bigskip

\noindent
by replacing the underlined $\wec{u}$'s with $\wec{v}$'s.
In short, if the two elements of the first column
of the second matrix in (\ref{matrixEQ})
are $\theta$-related, then the two elements in the parallel
column must also be $\theta$-related.
This, together with the earlier condition $q\equiv_{\theta} r$, which
asserts that the elements along the cross diagonal
are $\theta$-related, yields
that all elements of the matrix are $\theta$-related.
(That is, $p\equiv_{\theta} r\equiv_{\theta} q\equiv_{\theta} s$).
This is
in contradiction to the condition $r\not\equiv_{\theta} s$.
What we have contradicted was the assumption that the elements
$p$ and $r$, which come from the same column of the second
matrix in (\ref{matrixEQ}), are $\theta$-related.
\cqed

  \bigskip

Next, let
$D = \{(z,z)\;|\;z\in B\}$ be the diagonal subuniverse of $\m b^2$.
Let $\m c\leq \m b^2$ be the subalgebra of $\m b^2$
generated by $D$ and all pairs
$(u_i,v_i)$, $i=1,\ldots,n$, where $\wec{u}=(u_1,\ldots,u_n)$ and 
$\wec{v}=(v_1,\ldots,v_n)$ in the first matrix from (\ref{matrixEQ}).
Observe that the pairs $(p,q)$ and $(r,s)$ 
both belong to $\m c$, since
\[
(p,q)=
(t(\wec{a},\wec{u}), t(\wec{a},\wec{v}))
=t(\underbrace{(a_1,a_1),(a_2,a_2),\ldots,
(u_1,v_1),(u_2,v_2),\ldots}_{\text{generators of $\m c$}})\in C
\]
\bigskip

\noindent
and
\[
(r,s)=
(t(\wec{b},\wec{u}), t(\wec{b},\wec{v}))
=t(\underbrace{(b_1,b_1),(b_2,b_2),\ldots,
(u_1,v_1),(u_2,v_2),\ldots}_{\text{generators of $\m c$}})\in C.
\]

Let $\gamma$ be the principal congruence of $\m c$
generated by the single pair (of pairs)
$\big((p,q),(r,s)\big)$.

\begin{clm} \label{gamma_on_the_diagonal}
  The diagonal $D\subseteq C$ is a union of $\gamma$-classes.
  Moreover,
  if $\big((c,c), (d,d)\big)\in\gamma$, then
  $(c,d)\in\theta$.
\end{clm}

\cproof
Since $\gamma$ is the congruence generated by the pair
$\big((p,q), (r,s)\big)$,
it follows from
Maltsev's Congruence Generation Lemma that
to prove this claim it suffices to establish that
if $f$ is any unary polynomial of $\m c$,
and $f\big((p,q)\big) = (c,c)\in D$,
then $f\big((r,s)\big) = (d,d)\in D$
for some $d$ satisfying $(c,d)\in\theta$.

Assume that $f$ is a unary polynomial of $\m c$
and that $f\big((p,q)\big) = (c,c)$
for some $c$.
Unary polynomials of $\m c$ have the form
\[
f\big((x,y)\big)
= g\big((x,y),(\wec{u},\wec{v})\big)
= \big(g(x,\wec{u}), g(y,\wec{v})\big)
\]
\bigskip

\noindent
where $g(x,\wec{z})$ is a polynomial of $\m b$, since
$\m c$ is generated by $D$ and all pairs
$(u_i,v_i)$, $i=1,\ldots,n$, where $\wec{u}=(u_1,\ldots,u_n)$ and 
$\wec{v}=(v_1,\ldots,v_n)$. Thus $f\big((p,q)\big) = (c,c)$
can be rewritten
\begin{equation} \label{pANDq}
g(p,\wec{u}) = c = g(q,\wec{v})
\end{equation}
\bigskip

\noindent
for some polynomial $g$ of $\m b$. More fully, this is 
\[
g(t(\wec{a},\wec{u}),\wec{u}) = c = g(t(\wec{a},\wec{v}),\wec{v}).
\]
\bigskip

\noindent
Apply the term condition in the abelian algebra
$\m b$ to change the underlined
$\wec{a}$'s to $\wec{b}$'s below, so from 
\[
g(t(\underline{\wec{a}},\wec{u}),\wec{u}) = c = g(t(\underline{\wec{a}},\wec{v}),\wec{v})
\]
\bigskip

\noindent
we get
\[
g(t(\underline{\wec{b}},\wec{u}),\wec{u}) = g(t(\underline{\wec{b}},\wec{v}),\wec{v}).
\]

\bigskip

\noindent
Less fully, this equality may be rewritten
\begin{equation}\label{rANDs}
g(r,\wec{u}) = d = g(s,\wec{v})
\end{equation}
\bigskip

\noindent
for some $d$. In other words, the fact that $\m b$ is abelian
implies that if $f\big((p,q)\big) = (c,c)\in D$
for some $c$, then there is some $d$ such that
$f\big((r,s)\big) = \big(g(r,\wec{u}), g(s,\wec{v})\big)
= (d,d)$.

It remains to argue that we must have $(c,d)\in\theta$. Consider the
following $1,1$-matrix of $\m b$, where two of the entries
in the left matrix can be determined from Equation~(\ref{pANDq}):
\[
  \left[\begin{array}{cc}
g(p,\wec{u}) & g(p,\wec{v}) \\
g(q,\wec{u}) & g(q,\wec{v}) 
      \end{array}\right]=
  \left[\begin{array}{cc}
      c&*\\
      *&c
      \end{array}\right].
  \]
\bigskip
  
\noindent
The off-diagonal entries can be determined from the
diagonal entries, since $\m b$ is strongly abelian.
Namely, all four entries must equal $c$,
yielding $g(p,\wec{u})=g(p,\wec{v})=g(q,\wec{u})=g(q,\wec{v})=c$.
In particular, this and Equation~(\ref{rANDs}) yield
$(g(q,\wec{u}),g(r,\wec{u}))=(c,d)$. This shows that
the pair $(c,d)$ is a polynomial translate
of the pair $(q,r)$ via the translation
$x\mapsto g(x,\wec{u})$. Since $(q,r)\in\theta\in\Con(\m b)$,
according to the line after (\ref{matrixEQ}),
and $g(x,\wec{u})$ is a unary polynomial of $\m b$
it follows that $(c,d)\in\theta$.
\cqed

  \bigskip

\begin{clm}
$\m c/\gamma$ is not abelian.
\end{clm}

\cproof
We argue that $C(1,1;\gamma)$ fails in $\m c$.
For this it suffices to write down a bad $1,1$-matrix
of $\m c$:

\begin{equation} \label{matrixEQ2}
  \left[\begin{array}{cc}
 t\big((\wec{a},\wec{a}),\underline{(\wec{u},\wec{v})}\big)&
 t\big((\wec{b},\wec{b}),\underline{(\wec{u},\wec{v})}\big)\\
 \vphantom{A}&\\
 t\big((\wec{a},\wec{a}),\underline{(\wec{u},\wec{u})}\big)&
t\big((\wec{b},\wec{b}),\underline{(\wec{u},\wec{u})}\big)\\
      \end{array}\right]=
  \left[\begin{array}{cc}
      (p, q)&(r, s)\\
      \vphantom{A}&\\
     (p, p)&(r, r)\\
      \end{array}\right].
\end{equation}

(One should check that this \underline{is} a
$1,1$-matrix in $\m c$, i.e., that $t$ is being
applied to elements of $\m c$: $(a_i,a_i),
(b_j,b_j), (u_k,u_k),
(u_{\ell},v_{\ell})\in C$.)

Our goal is to show that
the first matrix in (\ref{matrixEQ2})
expresses a failure of $C(1,1;\gamma)$, which in more
standard notation might be written:

\[
t\big((\wec{a},\wec{a}),\underline{(\wec{u},\wec{v})}\big)
\equiv_{\gamma}
t\big((\wec{b},\wec{b}),\underline{(\wec{u},\wec{v})}\big),
\]

\bigskip

\noindent
while changing the underlined
$(\wec{u},\wec{v})$'s to 
$(\wec{u},\wec{u})$'s produces

\[
t\big((\wec{a},\wec{a}),\underline{(\wec{u},\wec{u})}\big)
\not\equiv_{\gamma}
t\big((\wec{b},\wec{b}),\underline{(\wec{u},\wec{u})}\big).
\]

To see that this truly is a failure of $C(1,1;\gamma)$,
notice that the elements on the first row of the second
matrix in (\ref{matrixEQ2}) are indeed $\gamma$-related,
since $\gamma$ was defined to be the congruence
of $\m c$ generated by
$\big((p,q), (r,s)\big)$.
But notice also that the elements on the second row of the second
matrix in (\ref{matrixEQ2}) cannot possibly be $\gamma$-related.
For, we proved in Claim~\ref{gamma_on_the_diagonal}
that $\gamma$-related pairs of the form
$\big((c,c),(d,d)\big)$
must satisfy $(c,d)\in\theta$, and
we proved that $(p,r)\notin \theta$ in Claim~\ref{theta-related}.
Hence $\big((p,p),(r,r)\big)\notin\gamma$.
\cqed

\bigskip

To summarize,
in the main portion of the proof we showed
that when $\m b$ is strongly abelian,
$\theta\in\Con(\m b)$, and $\m b/\theta$ is abelian but
not strongly rectangular, then there is an algebra
$\m c/\gamma\in\Hom\Sub(\m b^2)$
that is not abelian.
It follows that when $\m a$ is strongly abelian,
$\Hom\Sub(\m a^2)$ consists of abelian algebras,
$\m b\leq \m a$, $\theta\in\Con(\m b)$,
and $\m b/\theta$ is not strongly rectangular,
then there is an algebra
$\m c/\gamma\in\Hom\Sub(\m b^2)\subseteq \Hom\Sub(\m a^2)$
that is not abelian. This was what we needed to
establish, as one can verify by consulting
the third paragraph following Claim~\ref{fake-mck-val-clm}.
\end{proof}

\begin{thm} \label{subclassK}
  Let $\mathcal V$ be an abelian variety. If 
  ${\mathcal K}$
  is any subclass of ${\mathcal V}$ that consists of
  strongly abelian algebras, then the subvariety of
  $\mathcal V$ generated by $\mathcal K$ is
  strongly abelian.
  (Equivalently, the subclass of all strongly abelian
  algebras in $\mathcal V$ is a subvariety of $\mathcal V$.)
\end{thm}

\begin{proof}
Since ${\mathcal K}$ consists of strongly abelian algebras,
and the property of being
strongly abelian is expressible
by a family of universal Horn sentences,
the class $\Sub\Prod({\mathcal K})$ of algebras isomorphic
to subalgebras of products
of members of $\mathcal K$ consists of strongly abelian algebras.
Any $\m a\in\Sub\Prod(\mathcal K)$ will be strongly
abelian, hence will have the property
that its finitely generated
subalgebras are strongly abelian.
Moreover, since $\m a\in\mathcal V$, and $\mathcal V$ is abelian,
the class $\Hom\Sub(\m a^2)$ will consist of abelian algebras.
By Lemma~\ref{fake-mck-val}, the class 
$\Hom\Sub(\m a)$ consists of strongly
abelian algebras. Since $\m a\in\Sub\Prod(\mathcal K)$
was arbitrary, this means that
$\Hom\Sub(\Sub\Prod(\mathcal K)) = \Hom\Sub\Prod(\mathcal K)$
consists of strongly abelian algebras.
Since $\Hom\Sub\Prod(\mathcal K)$ is the variety generated
by $\mathcal K$, the theorem is proved.
\end{proof}

\begin{remark}
For \emph{locally finite} varieties,
the result stated in Theorem~\ref{subclassK}
is due to Matt Valeriote.
It is known from Tame Congruence Theory that
if $\mathcal V$ is a locally finite variety,
then the subclass $\mathcal S$
of locally strongly solvable algebras
in $\mathcal V$ is a subvariety of $\mathcal V$.
Valeriote proved in \cite{valeriote}
that the following are equivalent
for any locally finite \emph{abelian} variety $\mathcal S$:
\begin{enumerate}
\item[(1)]  $\mathcal S$ is locally strongly solvable.
\item[(2)]  $\mathcal S$ is strongly abelian.
\end{enumerate}

Here is how you deduce
Theorem~\ref{subclassK} in the locally finite case
from the statements just made.
Assume that $\mathcal V$ is a locally finite abelian variety.
Let $\mathcal S\subseteq \mathcal V$ be the subvariety of $\mathcal V$
consisting of locally strongly solvable algebras in $\mathcal V$.
Let $\mathcal K\subseteq \mathcal V$ be the subclass
of strongly abelian algebras of $\mathcal V$.
Since a strongly abelian algebra is locally strongly solvable,
$\mathcal K\subseteq \mathcal S$.
Valeriote's Theorem applied to the subvariety $\mathcal S$
shows that $\mathcal S\subseteq \mathcal K$,
hence $\mathcal K=\mathcal S$,
which shows that $\mathcal K$ is a subvariety of $\mathcal V$.
\end{remark}

\begin{thm} \label{III}
If $\mathcal V$ is an abelian variety that is not affine, then
  $\mathcal V$ contains a nontrivial strongly abelian subvariety.
\end{thm}

\begin{proof}
  Theorem~\ref{II} shows that if $\mathcal V$
  is abelian but not affine, then there is some
  nontrivial algebra in $\m a\in \mathcal V$ that
  is strongly abelian. By Theorem~\ref{subclassK},
  the subvariety of $\mathcal V$ generated
  by $\m a$ is a nontrivial strongly abelian
  subvariety of $\mathcal V$.
\end{proof}

\begin{cor} \label{minvar}
  Any minimal abelian variety of algebras is affine
  or strongly abelian. $\Box$
\end{cor}

\bibliographystyle{plain}

\begin{thebibliography}{10}

\bibitem{kearnes-id-simple}
Keith A. Kearnes, 
{\it  Idempotent simple algebras.}
in Logic and Algebra,
Lecture Notes in Pure and Appl. Math. 180, 1996, 529--572.

\bibitem{kearnes-min-id}
Keith A. Kearnes, 
{\it Almost  all minimal idempotent varieties are congruence modular.}
Algebra Universalis {\bf 44} (2000), 39--45. 

\bibitem{kearnes-kiss}
Keith A. Kearnes and Emil W.\ Kiss, 
{\sl The shape of congruence lattices.}
Mem.\ Amer.\ Math.\ Soc.\ {\bf 222} (2013), no. 1046.

\bibitem{kearnes-kiss-szendrei}
  Keith A. Kearnes, Emil W. Kiss, and \'Agnes Szendrei,
  {\it Varieties whose finitely generated members are free.}
  Algebra Universalis {\bf 79} (2018), no. 1, 79:3.

\bibitem{kksz2}
  Keith A. Kearnes, Emil W. Kiss, and \'Agnes Szendrei,  
  {\it Minimal abelian varieties of algebras, II},
  manuscript.

\bibitem{kksz3}
  Keith A. Kearnes, Emil W. Kiss, and \'Agnes Szendrei,  
  {\it Minimal abelian varieties of algebras, III},
  manuscript.

\bibitem{kearnes-kiss-valeriote}  
Keith A.\  Kearnes, Emil W.\ Kiss, and Matthew A.\ Valeriote, 
  {\it Minimal sets and varieties.}
  Trans.\ Amer.\ Math.\ Soc.\ {\bf 350} (1998), no. 1, 1--41.

\bibitem{kearnes-szendrei1}  
Keith A. Kearnes and \'Agnes Szendrei,
{\it A characterization of minimal locally finite varieties.}
Trans.\ Amer.\ Math.\ Soc.\ {\bf 349} (1997), no. 5, 1749--1768.

\bibitem{kearnes-szendrei2}
  Keith A. Kearnes and \'Agnes Szendrei,
{\it The relationship between two commutators.}
Internat.\ J.\ Algebra Comput. {\bf 8} (1998), 497--531.

\bibitem{mckenzie-valeriote}
Ralph McKenzie and Matthew A.\ Valeriote,   
{\sl The structure of decidable locally finite varieties.}
Progress in Mathematics, {\bf 79}. Birkh\"auser Boston, Inc.,
Boston, MA, 1989. 

\bibitem{szendrei-survey}
  \'Agnes Szendrei,
{\it A survey on strictly simple algebras and minimal varieties.}
Universal algebra and quasigroup theory (Jadwisin, 1989), 209--239,
Res.\ Exp.\ Math., {\bf 19}, Heldermann, Berlin, 1992. 
          
\bibitem{szendrei1}  
  \'Agnes Szendrei, 
  {\it Strongly abelian minimal varieties.}
  Acta Sci. Math. (Szeged) {\bf 59} (1994), no.
1-2, 25--42.

\bibitem{szendrei2}  
  \'Agnes Szendrei, 
  {\it Maximal non-affine reducts of simple affine algebras.}
  Algebra Universalis {\bf 34}
(1995), no. 1, 144--174.

\bibitem{valeriote}  
Matthew Anthony Valeriote,
On decidable locally finite varieties
(universal algebra, abelian, multi-unary),
PhD. dissertation,
University of California, Berkeley. 1986.
139 pp.
\end{thebibliography}

\end{document}